\documentclass[12pt]{article}
\usepackage{amsfonts,amsmath,color,amsthm,amssymb}
\usepackage{latexsym}

\title{Sharp Concentration of Hitting Size for Random Set Systems}

\author{Jessie Deering, Anant Godbole, and William Jamieson\\
Department of Mathematics and Statistics\\ East Tennessee State University \and
Lucia Petito\\
Department of Biostatistics, University of California, Berkeley}
	
\begin{document}
	\maketitle
\def\l{\mathbb L}
\def\la{\lambda}
\def\lr{\left(}
\def\rr{\right)}
\def\lc{\left\{}
\def\rc{\right\}}
\def\vp{\varphi}
\def\ca{\mathcal A}
\def\cp{\mathcal P}
\def\b{\beta}
\def\a{\alpha}
\def\e{\mathbb E}
\def\p{\mathbb P}
\def\v{\mathbb V}
\def\lg{{\rm lg}}
\def\nm{{n\choose m}}
\def\mr{{m\choose r}}
\def\nmr{{{n-m}\choose{m-r}}}
\newtheorem{thm}{Theorem}
\newtheorem{lm}[thm]{Lemma}
\newtheorem{cor}[thm]{Corollary}
\newtheorem{rem}[thm]{Remark}
\newtheorem{exam}[thm]{Example}
\newtheorem{prop}[thm]{Proposition}
\newtheorem{defn}[thm]{Definition}
\newtheorem{cm}[thm]{Claim}

		\begin{abstract}
			Consider the random set system ${\ca }=R(n,p)$ of $[n]:=\{1,2, \dots n\}$, where ${\ca }=\{A_j:A_j \in {\cal P}([n]),\text{ and $A_j$ selected with probability $p=p_n$}\}$. A set $H \subseteq [n]$ is said to be a hitting set for ${\ca}$ if $\forall A_j \in {\ca}$ $\vert A_j \cap H\vert \geq 1$. The second moment method is used to exhibit the sharp concentration of the minimal size of $H$ for a variety of values of $p$. 
		\end{abstract}
		
	\section{Introduction and Motivation}  A set $D$ of vertices in a graph $G=(V,E)$ forms a {\it dominating set} of $G$ if each $v\in V$ is either in $D$ or adjacent to some $d\in D$.  The domination number $\gamma=\gamma(G)$ is the size of the smallest dominating set of $G$.  Given a graph of minimum degree $\delta$, it is proved, e.g., in Alon and Spencer \cite{as} that \begin{equation}\gamma(G)\le\frac{1+\ln(\delta+1)}{\delta+1}.\end{equation}
In a result of direct relevence to this paper, Weber \cite{w} proved in 1981 that the domination number of the random graph $G(n,p)$ is sharply concentrated w.h.p.~if $p$ is fixed.
This result was extended in \cite{wg} to the case $p=p_n\to0$, where a two point concentration was shown to hold for $\gamma(G(n,p_n))$ provided $p_n$ did not decay too rapidly; specifically, $p=1/\log\log n$ works in the above result; $p=1/\log n$ does not.

Given a $k$-uniform hypergraph $H=(V,E_k)$, a {\it transversal} is a collection $T$ of vertices such that each edge $e\in E_k$ intersects $T$ in at least one vertex.  We will denote the transversal number of $H$ by $\tau(H)$.
A transversal is also called a {\it hitting set}, particularly in the Computer Science literature, where it is more typical than not for edges to be of different sizes.  Accordingly, we will reserve the terminology ``transversal" for $k$-uniform hypergraphs, and ``hitting set" for the general case. In a result that echoes  (1), Alon \cite{a} proved that for a 
$k$-uniform hypergraph with $v$ vertices and $e$ edges, 
\[\tau(H)\le(1+o(1))\frac{\log k}{k}(v+e)\enspace(k\to\infty).\]    The Computer Science literature has focused more on complexity issues for hitting sets; see, e.g., \cite{ckmv}, \cite{bhs}, \cite{lj}, and \cite{ers}.  The connection between {\it total domination} and transversals has been explored in \cite{ty}.

If all our edges are of cardinality two, i.e. if we have a graph, let $s\not\in T$, $T$ a transversal.  Then the only edges containing $s$, must be between $s$ and $t$ for $t\in T$.  Thus $T$ is a minimal hitting set iff $T^C$ is maximal independent.  Note that this is also true for arbitrary hypergraphs if independent sets are defined as collections of vertices for which there is no edge that is a subset of these vertices.  Also, the sharp  two point concentration of the maximal independent set in a random graph has been well understood since the early work of Bollob\'as and Erd\H os \cite{be} and Matula \cite {m}, and others.  
In these results on finite point concentration, nothing more than the second moment method was used, though more sophisticated machinery was employed by Alon and Krivelevich \cite{ak}, and Achlioptas and Naor \cite {an} to show the sharp concentration of the chromatic number of $G(n,p)$.  It will turn out that elementary methods will suffice in this paper; we will investigate the sharp concentration of the size of minimal hitting sets (or {\it hitting number}) for non-uniform hypergraphs.   

Our model consists of picking each set $A\subseteq\{1,2,\ldots,n\}$ with probability $p=p_n$.  
Let ${\ca}$ be the ensemble of picked sets, which we will call a {\it random set system} and denote by $R(n,p)$ (to mirror the $G(n,p)$ notation for a random graph).  The goal is to discover a class of $p$s for which the hitting number is close to the intuitive guess of $\lg(p\cdot2^n)$, where throughout this paper $\lg = \log_2$.  In Section 2, we set the stage for when a one or two point concentration holds for the hitting number, and, in Sections 3 and 4, details are provided for two canonical cases, namely those corresponding to $p=1/2^{n\b}$ and $p=n^\a/2^n$.

\section{Setting up the Two-Point Concentration}    Define the baseline random variable, $X_m$, to be the number of hitting sets of size $m$.  We start by exploring a lower bound on $\vert H\vert$.  
Clearly $$\e(X_m)={n\choose m}(1-p)^{2^{n-m}}\le{n\choose m}\exp\{-p2^{n-m}\},$$ since a set of size $m$ is hitting iff we do not pick any of the subsets of its complement to be in the random set system (actually we cannot by definition hit the empty set, so the correct exponent ought to be $2^{n-m}-1$). 
Let us set (with hindsight) $m=\lg(p\cdot2^n)-\vp(n)$.\footnote{In this paper we will encounter several functions that play a ``generic" role.  Examples of these functions are $\omega(n)$, $\vp(n),\epsilon_n$, and $\mu_n$.  They are each defined differently  in various parts of the paper, but their {\it role} is always the same, e.g. $\vp(n)$ will {\it always} denote how much smaller the hitting set size is than $\lg p\cdot 2^n$ and $\omega(n)$ will always be a function that tends to infinity at an arbitrarily slow rate.}
Thus \begin{equation}\p(X_m\ge 1)\le\e(X_m)\le{n\choose m}\exp\{-2^{\vp(n)}\}\to0\end{equation}
provided that ${n\choose m}\ll\exp\{2^{\vp(n)}\}$, 
and, using the inequality $1-p\ge e^{-p/(1-p)}$, with $\epsilon_n=\lg (1-p)^{-1}=O(p)$,
\begin{equation}\e(X_m)={n\choose m}(1-p)^{2^{n-m}}\ge{n\choose m}\exp\{-p2^{n-m+\varepsilon_n}\}\to\infty\end{equation} if ${n\choose m}\gg\exp\{2^{\vp(n)+\epsilon_n}\},$  where the $\vp$ functions in (2) and (3) are different.  Since zero-one probability thresholds often occur precisely where the associated expected value transitions from zero to infinity,
we anticipate that Equations (2) and (3) occur with near-consecutive values of $m$.

By Chebychev's inequality, $\p(X_m=0)\le\frac{\v(X_m)}{\e^2(X_m)}$, so to establish an upper bound on $\vert H\vert$ it would suffice to show that the variance is an order of magnitude smaller than the square of the mean whenever $m\ge m_0$ -- for some $m_0$ to be determined.  Since
$X_m=\sum_{j=1}^\nm I_j$, where the indicator variable $I_j$ equals one iff the $j$th $m$-set hits $R(n,p)$, we have that

\begin{eqnarray*}\v(X_m)&=&\e(X_m^2)-\e^2(X_m)\\
&=&\sum_{j=1}^{{n\choose m}}\e(I_j^2)-\lr\sum_{j=1}^{{n\choose m}}\e(I_j)\rr^2+\sum_{j\ne k}\e(I_jI_k)\\
&=&\e(X_m)-\e^2(X_m)+\sum_{j\ne k}\e(I_jI_k),\end{eqnarray*}
so that
\begin{equation}
\frac{\v(X_m)}{\la^2}=\frac{1}{\la}-1+\frac{\sum_{j\ne k}\e(I_jI_k)}{\la^2},
\end{equation}
where $\la=\la_m=\e(X_m)$.  Now two sets $A,B$ of size $m$ that intersect in $r$ elements both hit $\ca$ iff we do not pick, as part of $\ca$, any set that is a subset of $A^C$ or a subset of $B^C$; there are $2^{n-m}+2^{n-m}-2^{n-2m+r}$ of these.  Thus, substituting $s=m-r$ and assuming that $\la\ge1$, we have
\begin{eqnarray}\sum_{j\ne k}\e(I_jI_k)&=&\nm^2\cdot\sum_{r=0}^{m-1}\frac{\mr\nmr}{\nm}(1-p)^{2^{n-m+1}-2^{n-2m+r}}\nonumber\\
&=&\la^2\sum_{s=1}^{m}\la^{-2^{-s}}{m\choose s}{{n-m}\choose{s}}{n\choose m}^{2^{-s}-1},\nonumber\\
&\le&\la^2\sum_{s=1}^{m}{m\choose s}{{n-m}\choose{s}}{n\choose m}^{2^{-s}-1}.
\end{eqnarray}
By (4) and (5) it thus suffices to show that 
\begin{equation}
\sum_{s\ge1}{m\choose s}{{n-m}\choose{s}}{n\choose m}^{2^{-s}-1}=1+o(1)
\end{equation}
as $\la\to\infty$; this is really a simple statement about the function $m=m(n)$ as $n\to\infty$.  Let us set up what it takes to make (6) occur:  We first define, with $s_0=2(\lg(m\log n))$, the sums 
$$\Sigma_1=\sum_{s\ge s_0}{m\choose s}{{n-m}\choose{s}}{n\choose m}^{2^{-s}-1}$$ and $$\Sigma_2=\sum_{1\le s\le s_0-1}{m\choose s}{{n-m}\choose{s}}{n\choose m}^{2^{-s}-1}.$$
In $\Sigma_1$, we first bound as follows:
\[{{n}\choose{m}}^{2^{-s}}\le\lr\frac{ne}{m}\rr^{m/2^s}\le\lr\frac{ne}{m}\rr^{\frac{1}{m(\log n)^2}}=1+o(1),\]
so that
$$\Sigma_1\le (1+o(1))\sum_{s\ge s_0+1}\frac{{m\choose {m-s}}{{n-m}\choose{s}}}{{n\choose m}}=1+o(1),$$
since the sum above represents almost entirely the mass of a hypergeometric variable with mean $\sim m$, provided that $s_0\ll m$ -- which holds if $m\ge\Omega(\log\log n)$.  Turning to $\Sigma_2$, 
we have
\begin{eqnarray}
\Sigma_2&\le&\sum_{1\le s\le s_0-1}{m\choose s}{{n-m}\choose s}{n\choose m}^{-1/2}\nonumber\\
&\le& \sum_{1\le s\le s_0-1} n^{2s}(n/m)^{-m/2}\nonumber\\
&\le& n^{2s_0-m/2}m^{m/2}\nonumber\\
&=&e^{(2s_0-m/2)\log n+(m/2)\log m}.
\end{eqnarray}
(where we used the bounds $\max\{{m\choose s},{{n-m}\choose s}\}\le n^s$; ${n\choose m}\ge(n/m)^m$ in the second display above.)  We wish the estimate in (7) to be of magnitude $o(1)$ and thus need
\begin{equation}\log m<\lr1-\frac{8\lg(m\log n)}{m}\rr\log n.
\end{equation}
It is not too hard to check that (8) holds if $m$ is not too small or too large; specifically one needs
\begin{equation}
\Omega(\log\log n)\le m\le n-\Omega(\log n). 
\end{equation}
So, for $m$s satisfying (9), we get that the hitting size is at least $m+1$ if $\la=\la_m\to0$, while if $\la=\la_m\to\infty$, then the hitting size is at most $m$.  We next note that
\[\la^2_{m+1}={{n}\choose{m+1}}^2(1-p)^{2^{n-m}}={{n}\choose{m+1}}^2{n\choose m}^{-1}\la_m\gg\la_m,\]
certainly for all $m\in[1,n-3]$.  This leads 
to the conclusion that either $\la_m\to0$ or $\la_{m+1}\to\infty$.  If both these hold, then $|H|=m+1$ w.h.p.; on the other hand if $\la_{m-1}\to0; \la_m\to K; \la_{m+1}\to\infty$, or $\la_{m}\to0; \la_{m+1}\to K; \la_{m+2}\to\infty$ for some $K\in{\mathbb R}^+$, then we have a two point concentration.  We summarize the findings of this section in the following result:
\begin{thm}  Consider the random set system $\ca=R(n,p)$, where $p$ is unspecified.  Let ${\cal F}$ denote the interval $[\Omega(\log\log n),n-\Omega(\log n)]$, where the constants in the $\Omega$ functions can be readily specified.  Let $\ell=\sup\{m=m_n:\lim\e(X_m)=0\}$ and $h=\inf\{m:\lim\e(X_m)=\infty\}$.  Then, for suitable $p=p_n, \ell,h\in{\cal F}$; $h-\ell\in\{1,2\}$ and $\vert H\vert=\ell+1$ or $\vert H\vert=h$ w.h.p.
\end{thm}

It remains to solve for $m$ in terms of $p$.  In the next two sections, we consider the ``dense" case, where the hitting size is comparable to $n$ and the ``sparse" case, where we will seek to hit a system $\ca$ of size satisfying $\vert\ca\vert^{1/n}\to1$.  For specificity we use the values $p=1/2^{n\b}; 0<\b<1$ and $p=n^\a/2^n;\a>0$ respectively, even though other choices could have been made, with the analysis being quite similar.  In both sections, we seek to find a value of $m=m(p)$ for which $\e(X_m)\to 0$ and $\e(X_{m+1})\to\infty$ (or $\e(X_{m+2})\to\infty$).  
\section{A Dense Case, $p=1/2^{n\b}, 0<\b<1$.}
With $p=1/2^{n\b}$ and $m=(1-\b)n-\vp(n)$, where we restrict $\vp(n)\le\lg n$, $$\e(X_m)\le{n\choose m}\exp\{-2^{\vp(n)}\}\le{{n}\choose{\b n}}\lr\frac{1-\b}{\b}\rr^{\vp(n)}\exp\{-2^{\vp(n)}\}.$$  Stirling's formula next yields
\begin{eqnarray*}\e(X_m)&\le&\frac{C}{\sqrt n}(\b^{-\b}(1-\b)^{-(1-\b)})^n\lr\frac{1-\b}{\b}\rr^{\vp(n)}\exp\{-2^{\vp(n)}\}\\&\le&\frac{C}{\sqrt n}\gamma^{\lg n}\delta^n\exp\{-2^{\vp(n)}\},\end{eqnarray*} where $C$ is a universal constant, $\gamma=\max\{1,\frac{1-\b}{\b}\}$, and $\delta:=\b^{-\b}(1-\b)^{\b-1}\le2$.  We thus see that
$$\p(X_m\ge1)\le\e(X_m)\to0$$ if $$2^{\vp(n)}=n\ln\delta-\frac{1}{2}\ln n+(\ln\gamma)(\lg n)+\ln\omega(n)+\ln C,$$ or if
\begin{eqnarray*}\vp(n)&=&\lg\lr n\ln\delta-\frac{1}{2}\ln n+(\ln\gamma)(\lg n)+\ln\omega(n)\rr\\
&=&\lg(n\ln\delta)+o(1).\end{eqnarray*}  This yields
\begin{equation}\vert H\vert\ge\lfloor(1-\b)n-\lg( n\ln\delta)-o(1)\rfloor+1,\end{equation} where in (10), $\vp(n)\asymp\lg (n\ln\delta)\le \lg n$ as stipulated.

For the lower bound, we argue as follows:
$$\e(X_m)={n\choose m}(1-p)^{2^{n-m}}\ge{n\choose m}\exp\{-p2^{n-m+\varepsilon_n}\},$$ where $\epsilon_n=\lg (1-p)^{-1}=O(p)$.
Setting $m=(1-\b)n-\vp(n)$ (we are in search of a different $\vp(n)$ than in (10)) yields
\[\e(X_m)\ge{{n}\choose {\b n}}\exp\{-p2^{n-m+\varepsilon_n}\}\] 
if $\b<1/2$, and
$$\e(X_m)\ge{{n}\choose {\b n}}\lr\frac{1-\b}{2\b}\rr^{\lg n}\exp\{-p2^{n-m+\varepsilon_n}\}$$
if $\b\ge1/2$.
Simplifying as before we get $\e(X_m)\to\infty$ if
\begin{eqnarray*}\vp(n)&=&\lg\lr n\ln\delta-\frac{1}{2}\ln n+({\ln\eta})(\lg n)-\ln\omega(n)\rr{-\varepsilon_n}\\
&=&\lg(n\ln\delta)-o^*(1),\end{eqnarray*}where $\eta=\min\{\frac{1-\b}{2\b},1\}$,
and thus
\begin{equation}
\vert H\vert\le\lceil(1-\b)n-\lg( n\ln\delta)+o^*(1)\rceil.
\end{equation}
It is easy to verify that the $\vp(n)$ functions in (10) and (11) differ by $o(1)$.  Thus the worst case scenario is when these quantities  straddle an integer, when we have a two point concentration.  In the other case, we have that $\vert H\vert$ is a constant w.h.p.

We have proved \begin{thm}
Let $H=H(n,\b)$ be the size of the minimal hitting set of the random set system $R(n,1/2^{n\b})$ consisting of the ensemble that is generated when each set in $\cp([n])$ is independently picked with probability $p=2^{-n\b}$.  Then with probability approaching unity, $\vert H\vert = h$ or $h+1$, where 
\[h=\lfloor(1-\b)n-\lg(n\ln\delta)-o(1)\rfloor+1,\]
and where the $o(1)$ is as in the argument leading to (10).
\end{thm}
The next result follows immediately:
\begin{cor}
Let $I=I(n,\b)$ be the size of the maximal independent set of the random set system $R(n,1/2^{n\b})$ Then w.h.p., $\vert I\vert = i$ or $i+1$, where 
\[i=n-2-\lfloor(1-\b)n-\lg( n\ln\delta)-o(1)\rfloor.\]
\end{cor}

\section{A Sparse Case, $p=n^\a/2^n, \a>0$}  We now move on to the case $p=n^\a/2^n, \a>0$.  Notice that as in Theorem 2, the hitting number works out to be a {\it just a little smaller} than the value $\lg\e\vert\ca\vert=\lg (p\cdot 2^n)$, which can easily be seen to be the least $m$ such that the set $\{1,2,\ldots,m\}$ is {\it expected} to hit all the sets in $\ca$. 
\begin{thm} Let $H=H(n,\a)$ be the size of the minimal hitting set of the random set system $R(n,n^\a/2^{n})$, $\a>0$. Then with high probability, $\vert H\vert = h$ or $h+1$, where 
\[h=\lfloor\a\lg n-\lg(\a\lg n\ln n)-o(1)\rfloor+1.\]
\end{thm}
\begin{proof} Let $X_m$ be as before.  We have 
\[\e(X_m)\le{n\choose m}\exp\{-p2^{n-m}\}\le\lr\frac{ne}{m}\rr ^m\exp\{-p2^{n-m}\},\]
so, setting $p=n^\alpha/2^n$ and $m=\a\lg n-\vp(n)$ (where we restrict by seeking solutions with $\vp(n)\le3\lg(\lg n)$), we get
\[\e(X_m)\le\lr\frac{ne}{\a\lg n-3\lg\lg n}\rr^{\a\lg n}\exp\{-2^{\vp(n)}\}\to0\]
if
\[2^{\vp(n)}=\a\lg n\bigg(\ln n+1-\ln(\a\lg n-3\lg\lg n)\bigg)+\ln\omega(n)\]
or
\begin{eqnarray}\vp(n)&=&\lg\bigg(\a\lg n\big(\ln n+1-\ln(\a\lg n-3\lg\lg n)\big)+\ln\omega(n)\bigg)\nonumber\\
&=&\lg(\a\lg n\ln n)+o(1).\end{eqnarray}
Note that $\vp(n)\le3\lg\lg n$ in (12) if $n$ is sufficiently large.  Thus 
 \begin{equation}\vert H\vert\ge\lfloor\a\lg n-\lg(\a\lg n\ln n)-o(1)\rfloor+1.\end{equation}

Next, setting $\vp(n)=\lg[\a(1-\mu_n)\lg n\ln n]-\epsilon_n$, where $\mu_n=\frac{5+\a}{\a}\frac{\lg\lg n}{\lg n}$, and $m=\a\lg n-\vp(n)$, we see that
\begin{eqnarray*}
\e(X_m)&=&{n\choose m}(1-p)^{2^{n-m}}\\
&\ge&{n\choose m}\exp\{-2^{\vp(n)+\epsilon_n}\}\\
&=&\frac{{n\choose m}}{n^{\a(1-\mu_n)\lg n}}\\
&\ge&\frac{(n-m)^m}{C\sqrt{m}}\lr\frac{e}{m}\rr^m\frac{1}{n^{\a(1-\mu_n)\lg n}}\\
&\ge&\exp\{-m^2/(n-m)\}\frac{1}{C\sqrt{m}}\lr\frac{ne}{m}\rr^m\frac{1}{n^{\a(1-\mu_n)\lg n}}\\
&\ge&\frac{1}{2C\sqrt{m}}\lr\frac{ne}{\a\lg n}\rr^{\a\lg n-\vp(n)}\frac{1}{n^{\a(1-\mu_n)\lg n}}\\
&\ge&\frac{1}{2C\sqrt{m}}\frac{n^{\a\mu_n\lg n-\vp(n)}}{(\a\lg n)^{\a\lg n-\vp(n)}}\\
&\ge&\frac{1}{2C\sqrt{m}}\frac{n^{(5+\a)\lg\lg n-3\lg\lg n}}{(\a\lg n)^{\a\lg n-\vp(n)}}\\
&\ge&n^{\lg\lg n}\\
&\to&\infty.
\end{eqnarray*}
Together with (13), this completes the proof of Theorem 4.  
\end{proof}

\section{Open Questions} We feel that deriving similar concentrations for hitting set size of random uniform hypergraphs would be of value,  as would be results in which subsets of various sizes are picked with (a wide variety of) size-biased probabilities.  
\section{Acknowledgments}  The research of all the authors was supported by NSF Grant 1004624.  We thank the referee of a previous version of this paper, whose suggestions have greatly improved and streamlined the paper.


\begin{thebibliography}{99}
\bibitem{an} D.~Achlioptas and A.~Naor (2005).  ``The two possible values of the chromatic number of a random graph," {\it Ann.~Math.} {\bf 162,}  1335--1351.
\bibitem{a} N.~Alon (1990).  ``Transversal numbers of uniform hypergraphs," {\it Graphs and Combinatorics} {\bf 6}, 1--4.
\bibitem{ak} N.~Alon and M.~Krivelevich (1997). ``The concentration of the chromatic number of random graphs,"  {\it Combinatorica} {\bf 17}, 303--313.
\bibitem{as} N.~Alon and J.~Spencer (1992).  {\it The Probabilistic Method.} Wiley, New York.	
\bibitem{bhs}  M.~Bl\"aser, M.~Hardt, and D.~Steurer (2008).  ``Asymptotically Optimal Hitting Sets Against Polynomials," in ICALP '08: {\it Proceedings of the 35th International Colloquium on Automata, Languages and Programming, Part I,} pp.~345--356, Springer Verlag, New York.
\bibitem{be} B.~Bollob\'as and P.~Erd\H os (1976).  ``Cliques in random graphs," {\it Math.~Proc.~Camb.~Phil.~Soc.} {\bf 80,} 419--427.
\bibitem{ckmv} K.~Chandrasekaran, R.~Karp, E.~Moreno-Centeno, and S.~Vempala (2011). ``Algorithms for Implicit Hitting Set Problems," 
Accepted by ACM-SIAM Symposium on Discrete Algorithms (SODA11).
\bibitem{ers} G.~Even, D.~Rawitz, and S.~Shahar (2005). ``Hitting sets when the VC-dimension is small,"
{\it Information Processing Letters} {\bf 95}, 358--362.  
\bibitem{lj} L.~Li and Y.~Jiang (2002). ``Computing Minimal Hitting Sets with Genetic
Algorithm," in {\it Proceedings of the 13th International Workshop on Principles of Diagnosis}, 77--80.
\bibitem{m} D.~Matula (1976). ``The largest clique size in a random graph," Technical Report, Department of Computer Science, Southern Methodist University, Dallas, Texas.
\bibitem{ty} S.~Thomass\'e and A.~Yeo (2007).  ``Total domination of graphs and small transversals of hypergraphs," {\it Combinatorica} {\bf 27}, 473--487.
\bibitem{w}  K.~Weber (1981). ``Domination number for almost every graph," {\it Rostock. Math. Kolloq.}
{\bf 16}, 31--43.
\bibitem{wg} B.~Wieland and A.~Godbole (2001). 	``On the domination number of a random graph," {\it Electr.~J.~Combinatorics} {\bf 8,} Paper R37, 13 pages.



\end{thebibliography}
\end{document}